\documentclass{amsart}
\usepackage{latexsym,amssymb,amsmath,amsthm,amscd,graphicx}
\makeatletter
\@namedef{subjclassname@2010}{%
  \textup{2010} Mathematics Subject Classification}
\makeatother

\numberwithin{equation}{section}
\newtheorem{theorem}{Theorem}[section]
\newtheorem{lemma}[theorem]{Lemma}
\newtheorem{corollary}[theorem]{Corollary}

\theoremstyle{definition}

\newtheorem{remark}[theorem]{Remark}
\newtheorem{definition}[theorem]{Definition}

\theoremstyle{remark}

\newcommand{\cA}{{\mathcal A}}

\newcommand{\cE}{{\mathcal E}}
\newcommand{\cF}{{\mathcal F}}
\newcommand{\cG}{{\mathcal G}}

\newcommand{\cL}{{\mathcal L}}

\newcommand{\cP}{{\mathcal P}}
\newcommand{\cW}{{\mathcal W}}

\newcommand{\N}{{\mathbb N}}

\newcommand{\R}{{\mathbb R}}

\newcommand{\Z}{{\mathbb Z}}

\def\al{\alpha}

\def\kp{\kappa}
\def\la{\lambda}
\def\sg{\sigma}

\def\l{\left}
\def\r{\right}
\def\ds{\displaystyle}
\def\ol{\overline}

\begin{document}

\title
[Positive operators and maximal operators in a filtered measure space]
{Positive operators and maximal operators in a filtered measure space}
\author[H.~Tanaka]{Hitoshi Tanaka}
\address{Graduate School of Mathematical Sciences, The University of Tokyo, Tokyo, 153-8914, Japan}
\email{htanaka@ms.u-tokyo.ac.jp}
\thanks{
The first author is supported by 
the Global COE program at Graduate School of Mathematical Sciences, the University of Tokyo, 
Grant-in-Aid for Scientific Research (C) (No.~23540187), 
the Japan Society for the Promotion of Science, 
and was supported by F\=ujyukai foundation.
}
\author[Y.~Terasawa]{Yutaka Terasawa}
\address{Graduate School of Mathematical Sciences, The University of Tokyo, Tokyo, 153-8914, Japan}
\email{yutaka@ms.u-tokyo.ac.jp}
\thanks{
The second author is a Research Fellow of the Japan Society for the Promotion of Science.
}
\date{}

\begin{abstract}
In a filtered measure space, 
a characterization of weights for which the trace inequality of a positive operator holds is given
by the use of discrete Wolff's potential. 
A refinement of the Carleson embedding theorem is also introduced. 
Sawyer type characterization of weights 
for which a two-weight norm inequality for a 
generalized Doob's maximal operator holds 
is established by an application of our Carleson embedding theorem. 
Moreover, 
Hyt\"{o}nen-P\'{e}rez type
one-weight norm estimate for Doob's maximal operator 
is obtained by the use of our two-weight characterization. 
\end{abstract}

\maketitle

\section{Introduction}\label{sec1}
Weighted Norm Inequalities in Harmonic analysis is an old subject whose systematic 
investigation was initiated by \cite{Mu}, \cite{CoFe} and \cite{MuWh} etc.. A classical reference in the field is \cite{GaRu}.

Dyadic Harmonic Analysis has recently
acquired a renewed attention because of its wide applicability to Classical Harmonic Analysis,
including weighted norm inequalities. 
Petermichl \cite{Pe} and Nazarov-Treil-Volberg \cite{NaTrVo} 
were cornerstone works, 
whose investigations have been continued by many authors. 
This subject is also old, 
which can be found in \cite{Sa1} and \cite{GaJo} etc..  
For more complete references, 
we refer to the bibliographies of \cite{NaTrVo} and \cite{La}.

Two of the important topics in the intersection of these subjects are 
to get sharp one-weight estimates of usual operators in Classical Harmonic Analysis and 
to get necessary and sufficient conditions of weights for the boundedness of those operators in the two-weight setting.
Interestingly, these two topics are closely related. 
One way to attack these problems is a dyadic discretization technique. 
For the first problem, 
one of the important steps of a solution is 
getting a sharp one-weight estimate for 
a dyadic discretization of a singular integral operator, 
i.e., a generalized Haar shift operator. 
A sharp one-weight estimate of general singular integral operators, 
i.e., the $A_2$-conjecture, 
which has been an open problem in this field for a long time, 
was settled by Hyt\"{o}nen \cite{Hy3} along this line 
and its simpler proofs were found by several authors 
(cf. \cite{HyLaPe,Le2} etc.).
For (linear) positive operators, 
one example of which is a fractional integral operator, 
investigations along this line was done by several authors 
\cite{KeSa,SaWh,SaWhZh,VeWh,CaOrVe1,CaOrVe2}
and more recently by
\cite{LaMoPeTo,LaSaUr,Ka1,Ka2,Tr}. 
For the Hardy-Littlewood maximal operator (including a fractional maximal operator), 
Sawyer \cite{Sa1} got a two-weight characterization
by considering the dyadic Hardy-Littlewood (fractional) maximal operator.
Recently, 
using similar techniques, 
the sharp weighted estimates of the Hardy-Littlewood (fractional) maximal operator is established 
in the works 
\cite{Le1,LaMoPeTo,HyKa,HyPe},
which are continuations of the work of Buckley \cite{Bu}. 
For a survey of these developments, we refer to \cite{Per}, \cite{Hy4} and \cite{Hy5}. 

On the other hand, 
Martingale Harmonic Analysis is a subject which has also been well studied. 
Doob's maximal operator, 
which is a generalization of the dyadic Hardy-Littlewood maximal operator, 
and a martingale transform, 
which is an analogue of a singular integral in Classical Harmonic Analysis, 
are important tools in stochastic analysis. 
This field is called Martingale Harmonic Analysis and 
is well explained in the books 
by Dellacherie and Meyer \cite{DeMe}, 
Long \cite{Lo} and Kazamaki \cite{Ka}.
For Doob's maximal operator, 
one-weight estimate was studied first by Izumisawa and Kazamaki \cite{IzKa}, 
assuming some regularity condition on $A_p$ weights.
Later, Jawerth \cite{Ja} found that the added property is superfluous 
(see Remark \ref{rem4.6} below). 
For two-weight norm inequalities, 
the first study is done by Uchiyama \cite{Uc}, 
concerning necessary and sufficient condition of weights for weak type $(p,p)$ inequalities to hold.
Concerning strong $(p,q)$ type inequalities, 
Long and Peng \cite{LoPe} 
found necessary and sufficient conditions for weights, 
which is the analogous to Sawyer's condition for the boundedness of the Hardy-Littlewood maximal operator. 
There is also a recent work by Chen and Liu \cite{ChLi} on this topic.
For positive operators, 
there seems no work done in a filtered probability space or in a filtered measure space 
and we shall try to generalize 
the results in the Euclidean space of the weighted estimate for dyadic positive operators 
to those in a martingale setting.
(For fractional integral operators in a martingale setting, there is a recent work by Nakai and Sadasue \cite{NaSa}.)

The study of a boundedness property of positive operators and maximal operators 
is closely related to the Carleson embedding (or measure) theorem, 
which is a martingale analogue of 
the Carleson embedding theorem of a Hardy space into a weighted Lebesgue space.
In the dyadic setting in the Euclidean space, 
this coincides with the Dyadic Carleson embedding theorem.
The Carleson measure in a continuously filtered probability space was first introduced by Arai \cite{Ar}
with an application to the corona theorem on Complex Brownian Spaces.  
This was rediscovered later by Long \cite{Lo} in a discrete case, 
with an application to a characterization of $BMO$ martingales. 

Since a dyadic martingale is a special martingale in many ways, 
it might be useful to see which part of the theory of Dyadic Harmonic Analysis 
can be generalized to that of Martingale Harmonic Analysis, 
and which part is special to Dyadic Harmonic Analysis. 
Our contributions can be regarded as such an attempt.
We also expect that such results have some applications to 
stochastic analysis and analysis on metric spaces.

The purpose of this paper is to develop a theory of weights for 
positive operators and 
generalized Doob's maximal operators 
in a filtered measure space. 
Martingale Harmonic Analysis in a filtered (infinite) measure space is treated in 
\cite{St,Sc,Hy2,Ke,HyKe}.
In this contribution, 
we generalize the results of dyadic positive operators in the Euclidean space 
\cite{KeSa,SaWh,CaOrVe1,CaOrVe2}
to a filtered measure space.
The generalization of the results in 
\cite{LaSaUr} or \cite{Tr} 
to our setting seems difficult, 
since they use arguments related to an inclusion of cubes extensively.
We also investigate a necessary and sufficient condition of weights for a 
two-weight norm inequality of generalized Doob's maximal operator in a filtered measure space
which are generalization of both dyadic Hardy-Littlewood maximal operator and dyadic
fractional maximal operator. 
To state our main theorem, let us introduce some notations and terminologies, 
most of which are standard (cf. \cite{Hy2}). 

Let a triplet $(\Omega,\cF,\mu)$ be a measure space. 
Denote by $\cF^0$ the collection of sets in $\cF$ with finite measure. 
The measure space $(\Omega,\cF,\mu)$ is called $\sg$-finite 
if there exist sets $E_i\in\cF^0$ such that
$\bigcup_{i=0}^{\infty}E_i=\Omega$.
In this paper all measure spaces are assumed to be $\sg$-finite. 
Let $\cA\subset\cF^0$ be an arbitrary subset of $\cF^0$. 
An $\cF$-measurable function 
$f:\,\Omega\to\R$ 
is called $\cA$-integrable if 
it is integrable on all sets of $\cA$, 
i.e., 
$$
1_{E}f\in L^1(\cF,\mu)
\text{ for all }
E\in\cA.
$$
Denote the collection of all such functions by
$L_{\cA}^1(\cF,\mu)$.

If $\cG\subset\cF$ is another $\sg$-algebra, 
it is called a sub-$\sg$-algebra of $\cF$. 
A function 
$g\in L_{\cG^0}^1(\cG,\mu)$ 
is called the conditional expectation of 
$f\in L_{\cG^0}^1(\cF,\mu)$ 
with respect to $\cG$ if there holds
$$
\int_{G}f\,d\mu=\int_{G}g\,d\mu
\text{ for all }
G\in\cG^0.
$$
The conditional expectation of $f$ with respect to $\cG$ 
will be denoted by $E[f|\cG]$, 
which exists uniquely in 
$L_{\cG^0}^1(\cG,\mu)$ 
due to $\sg$-finiteness of $(\Omega,\cG,\mu)$.

A family of sub-$\sg$-algebras 
$(\cF_i)_{i\in\Z}$ 
is called a filtration of $\cF$ if 
$\cF_i\subset\cF_j\subset\cF$ 
whenever $i,j\in\Z$ and $i<j$. 
We call a quadruplet 
$(\Omega,\cF,\mu;(\cF_i)_{i\in\Z})$ 
a $\sg$-finite filtered measure space.
We write 
$$
\cL
:=
\bigcap_{i\in\Z}L_{\cF_i^0}^1(\cF,\mu).
$$
Notice that 
$L_{\cF_i^0}^1(\cF,\mu)
\supset
L_{\cF_j^0}^1(\cF,\mu)$ 
whenever $i<j$. 
For a function $f\in\cL$ 
we will denote $E[f|\cF_i]$ by $\cE_if$. 
By the tower rule of conditional expectations, 
a family of functions
$\cE_if\in L_{\cF_i^0}^1(\cF_i,\mu)$ 
becomes a martingale. (see Definition \ref{def2.1} below). 

By a weight we mean a nonnegative function 
which belongs to $\cL$ and, 
by a convention, 
we will denote the set of all weights by $\cL^{+}$. 

Let $\al_i$, $i\in\Z$, be a 
nonnegative bounded $\cF_i$-measurable function 
and set $\al=(\al_i)$. 
For a function $f\in\cL$ 
we define a positive operator $T_{\al}$ by 
$$
T_{\al}f:=\sum_{i\in\Z}\al_i\cE_if,
$$
and, define a generalized Doob's maximal operator $M_{\al}$ by 
$$
M_{\al}f:=\sup_{i\in\Z}\al_i|\cE_if|.
$$
When $\al=(1_{\Omega})$ this is Doob's maximal operator and 
we will write then $M_{\al}f=:f^{*}$. 

In this paper we shall first investigate the characterization of the weight 
$w\in\cL^{+}$ for which 
the trace inequality for the discrete positive operator $T_{\al}$ 
\begin{equation}\label{1.1}
\|T_{\al}f\|_{L^q(wd\mu)}
\le C_{\al,w}
\|f\|_{L^p(d\mu)}
\end{equation}
holds with $0<q<\infty$ and $1<p<\infty$.

In order to guess what the sufficient condition for \eqref{1.1} to hold is,
we argue heuristically in the following. 
We now assume that the inequality \eqref{1.1} holds 
for $1<p\le q<\infty$. Then, 
since the conditional expectation operators are selfadjoint, 
by duality there holds 
\begin{equation}\label{1.2}
\|T_{\al}(gw)\|_{L^{p'}(d\mu)}
\le C
\|g\|_{L^{q'}(wd\mu)},
\end{equation}
where $p'=\ds\frac{p}{p-1}$ is the conjugate exponent number of $p$. 
Following a principle of the weight theory, due to Sawyer \cite{Sa2}, 
to verify \eqref{1.1} it might suffice only to test 
\eqref{1.1} and \eqref{1.2} over 
the characteristic functions $1_{E}$. 
More precisely, one can expect that, 
the condition that 
\begin{equation}\label{1.3}
\l(\int_{E}\l(\sum_{j\ge i}\al_j\r)^qw\,d\mu\r)^{\frac1q}
\le C
\mu(E)^{\frac1p}
\end{equation}
and
\begin{equation}\label{1.4}
\l(\int_{E}\l(\sum_{j\ge i}\al_j\cE_jw\r)^{p'}\,d\mu\r)^{\frac1{p'}}
\le C
[wd\mu](E)^{\frac1{q'}}
\end{equation}
for any $E\in\cF_i^0$, $i\in\Z$, 
is sufficient for the inequality \eqref{1.2} to hold. 
This fact was verified for
positive operators associated the dyadic lattices in $\R^n$ 
\cite{LaSaUr} (and also \cite{Tr}). 

For some technical reasons, instead of the condition \eqref{1.3}, 
we must postulate the following strong condition \eqref{1.5} and then 
we shall prove that the condition \eqref{1.4} is sufficient for the inequality \eqref{1.2} to hold 
(cf. \cite{KeSa,SaWh} in the Euclidean space case).

\begin{quote}
The function $\al_i$, $i\in\Z$, 
is a nonnegative bounded $\cF_i$-measurable and 
$\ol{\al}_i\in\cL^{+}$, 
where 
$\ds\ol{\al}_i:=\sum_{j\ge i}\al_j$.
Moreover, 
\begin{equation}\label{1.5}
\cE_i\ol{\al}_i\approx\ol{\al}_i,
\end{equation}
holds.
\end{quote}

\begin{theorem}\label{thm1.1}
Let $1<p\le q<\infty$, 
$\al$ satisfy the condition {\rm\eqref{1.5}} 
and $w\in\cL^{+}$ be a weight. Then 
the following statements are equivalent:

\begin{enumerate}
\item[{\rm(a)}] 
There exists a constant $C_1>0$ such that 
$$
\|T_{\al}f\|_{L^q(wd\mu)}
\le C_1
\|f\|_{L^p(d\mu)};
$$
\item[{\rm(b)}] 
There exists a constant $C_2>0$ such that 
$$
\l(\int_{E}\l(\sum_{j\ge i}\al_j\cE_jw\r)^{p'}\,d\mu\r)^{\frac1{p'}}
\le C_2
[wd\mu](E)^{\frac1{q'}}
$$
for any $E\in\cF_i^0$, $i\in\Z.$
\end{enumerate}
Moreover,
the least possible $C_1$ and $C_2$ are equivalent.
\end{theorem}

In their papers \cite{CaOrVe1} and \cite{CaOrVe2}, 
Cascante, Ortega and Verbitsky 
established the characterization the weight $w$ for which the inequality \eqref{1.1} holds 
for $0<q<p<\infty$ and $1<p<\infty$ 
in terms of discrete Wolff's potential
in the cases when discrete positive integral operators are associated to the dyadic cubes in $\R^n$. 
The following theorem is an extension of their results to a filtered measure space. 
(cf. \cite{Ta,TaGu} in the Euclidean space).
Our condition \eqref{1.5} corresponds to 
\lq\lq the dyadic logarithmic bounded oscillation condition" 
introduced in \cite{CaOrVe1}.

\begin{theorem}\label{thm1.2}
Let 
$\al$ satisfy the condition {\rm\eqref{1.5}}, 
$w\in\cL^{+}$ be a weight and 
consider the following statements: 

\begin{enumerate}
\item[{\rm(a)}] 
There exists a constant $C_1>0$ such that 
$$
\|T_{\al}f\|_{L^q(wd\mu)}
\le C_1
\|f\|_{L^p(d\mu)};
$$
\item[{\rm(b)}] 
There exists a constant $C_2>0$ such that, 
for $\ds\frac1r=\frac1q-\frac1p$, 
$$
\|(\cW_{\al}[w])^{\frac1{p'}}\|_{L^r(wd\mu)}<c_2,
$$
where 
$$
\cW_{\al}[w]
:=
\sum_{i\in\Z}\al_i\ol{\al}_i^{p'-1}(\cE_iw)^{p'-1}
$$
is discrete Wolff's potential in a filtered measure space. 
\end{enumerate}

\noindent
Then, 
if $0<q<p<\infty$ and $1<p<\infty$, 
{\rm(b)} implies {\rm(a)} with $C_1\le CC_2$. 
Conversely, if $1<q<p<\infty$, 
{\rm(a)} implies {\rm(b)} with $C_2\le CC_1$. 
\end{theorem}

\begin{remark}\label{rem1.3}
In \cite{CaOrVe2}, 
in the cases when discrete positive integral operators are associated to the dyadic cubes in $\R^n$,
Cascante, Ortega and Verbitsky 
proved the equivalence between (a) and (b) 
in the full range $0<q<p<\infty$ and 
$1<p<\infty$. 
\end{remark}

Thanks to a powerful lemma (Lemma \ref{lem2.3} below) and the condition \eqref{1.5}, 
the proof of Theorems \ref{thm1.1} and \ref{thm1.2} 
can be reduced to the Carleson embedding (or measure) theorem. 
In Section \ref{sec3} 
we shall investigate that theorem in the setting of a filtered measure space 
(see Theorems \ref{thm3.1} and \ref{thm3.5}).
In Section \ref{sec4}, 
as an application of that theorem, 
we establish the analogue of Sawyer type characterization of weights 
for which two-weight norm inequality 
for the generalized Doob's maximal operator $M_{\al}$ holds 
(see Theorem \ref{thm4.1}). 
In Section \ref{sec5} 
we also establish Hyt\"{o}nen-P\'{e}rez type
one-weight norm estimate for Doob's maximal operator $f^{*}$ 
(see Theorem \ref{5.1}). 

Finally, we would like to comment on our weight class $\cL^{+}$. 

\begin{remark}\label{rem1.4}
Let 
$(\Omega,\cF,\mu;(\cF_i)_{i\in\Z})$ 
be a $\sg$-finite filtered measure space.
Then, it naturally contains a filtered probability space with a filtration indexed by $\N$ and 
a Euclidean space with a dyadic filtration. 
It also contains doubling metric measure space
with dyadic lattice constructed by Hyt\"{o}nen and Kairema \cite{HyKa}.
Our weight class $\cL^{+}$ coincides with 
a set of all locally integrable weights
in the case of the Euclidean space with a Lebesgue measure with a dyadic filtration. 
Since the dyadic $A_p$ weights in Euclidean space are locally integrable, 
it seems natural to introduce the class $\cL^{+}$. 
We could not find this class of weights 
in a filtered measure space in the literatures.
We notice that the class $L_{\cF^0}^1(\cF,\mu)$ 
used in several literatures does not include functions which grows at spacial infinity 
in the Euclidean space with $\cF$ a $\sigma$-algebra of the Lebesgue measurable sets
and $ \mu $ a Lebesgue measure.
\end{remark}

The letter $C$ will be used for constants
that may change from one occurrence to another.
Constants with subscripts, such as $C_1$, $C_2$, do not change
in different occurrences.
By $A\approx B$ we mean that 
$c^{-1}B\le A\le cB$ with some positive constant $c$ independent of appropriate quantities.

\section{Proof of Theorems 1.1 and 1.2}\label{sec2}
In what follows we prove Theorems \ref{thm1.1} and \ref{thm1.2}. 
We first list two basic properties of the conditional expectation 
and the definition of a martingale. 

Let $(\Omega,\cF,\mu)$ be a $\sg$-finite measure space 
and $\cG$ be a sub-$\sg$-algebra of $\cF$. 
Then the following holds.

\begin{enumerate}
\item[{\rm(i)}] 
Let 
$f\in L_{\cG^0}^1(\cF,\mu)$ 
and $g$ be a $\cG$-measurable function. 
Then the two conditions 
$fg\in L_{\cG^0}^1(\cF,\mu)$ 
and 
$gE[f|\cG]\in L_{\cG^0}^1(\cG,\mu)$ 
are equivalent and, 
assuming one of these conditions, we have 
$$
E[fg|\cG]=gE[f|\cG];
$$
\item[{\rm(ii)}] 
Let 
$f_1,f_2\in L_{\cG^0}^1(\cF,\mu)$. 
Then the three conditions 
$$
E[f_1|\cG]f_2\in L_{\cG^0}^1(\cG,\mu),\quad
E[f_1|\cG]E[f_2|\cG]\in L_{\cG^0}^1(\cG,\mu)
\text{ and }
f_1E[f_2|\cG]\in L_{\cG^0}^1(\cG,\mu)
$$
are all equivalent and, 
assuming one of these conditions, we have 
$$
E[E[f_1|\cG]f_2|\cG]
=
E[f_1|\cG]E[f_2|\cG]
=
E[f_1E[f_2|\cG]|\cG].
$$
\item[\rm{(iii)}]
Let $ \cG_1 \subset \cG_2~(\subset \cF) $ be two sub-$\sigma$-algebras of $\cF$ and let $ f \in L_{\cG_2^0}(\cF, \mu). $
Then 
$$
E[f | \cG_1] = E[E[f | \cG_2] | \cG_1].
$$
\end{enumerate}

(i) can be proved by an approximation by simple functions. 
The property (ii) means that 
conditional expectation operators are selfadjoint and
can be easily deduced from (i). (iii) can be proved easily
and called the tower rule of conditional expectations.

\begin{definition}\label{def2.1}
Let 
$(\Omega,\cF,\mu;(\cF_i)_{i\in\Z})$ 
be a $\sg$-finite filtered measure space.
Let $(f_i)_{i\in\Z}$ be a sequence of 
$\cF_i$-measurable functions. 
Then the sequence $(f_i)_{i\in\Z}$ is called a 
\lq\lq martingale" if 
$f_i\in L_{\cF_i^0}^1(\cF_i,\mu)$
and 
$f_i=\cE_if_j$ 
whenever $i<j$. 
\end{definition}

We also introduce the notion of a stopping time for later uses. 

\begin{definition}\label{def2.2}
Let 
$(\Omega,\cF,\mu;(\cF_i)_{i\in\Z})$ 
be a $\sg$-finite filtered measure space. 
Then a function 
$\tau:\,\Omega\rightarrow\{-\infty\}\cup\Z\cup\{+\infty\}$ 
is called a stopping time if for any $i\in\Z$ 
$$
\{\tau\le i\}
=
\{\omega\in\Omega:\,\tau(\omega)\le i\}
\in\cF_i.
$$
\end{definition}

Let $f_i$, $i\in\Z$, be an $\cF_i$-measurable function 
and let $\la\in\R$. Then, 
it is easy to see that 
$\tau:=\inf\{i:\,f_i>\la\}$ 
is a stopping time.
All the stopping times we will use are of this type.

Next we will state a principal lemma which plays a key role in the proof of Theorems \ref{thm1.1} and \ref{thm1.2}. 

\subsection{Principal lemma}\label{ssec2.1}
The following is the principal lemma, 
which is an extension of 
\cite[Theorem 2.1]{CaOrVe1} to a filtered measure space.

\begin{lemma}\label{lem2.3}
Let $\al_i$, $i\in\Z$, be a 
nonnegative bounded $\cF_i$-measurable function, 
let $s>1$ and $w\in\cL^{+}$ be a weight. Then 
the following quantities are equivalent: 
\begin{alignat*}{2}
A_1&:=\int_{\Omega}\l(\sum_{i\in\Z}\al_i\cE_iw\r)^s\,d\mu;
\\
A_2&:=\int_{\Omega}\sum_{i\in\Z}\al_i\cE_iw\l(\cE_i(\ol{\al}_iw)\r)^{s-1}\,d\mu;
\\
A_3&:=\int_{\Omega}\l(\sup_{i\in\Z}\cE_i(\ol{\al}_iw)\r)^s\,d\mu,
\end{alignat*}
where 
$\ds\ol{\al}_i:=\sum_{j\ge i}\al_j$.
\end{lemma}

\begin{proof}
By a standard limiting argument, 
we may assume without loss of generality that 
there are only a finite number of $\al_i\ne 0$ 
and $w$ is bounded and summable. 

\noindent{\bf (i)}\ \ 
We prove $A_1\le CA_2$. 
We use an elementary inequality
\begin{equation}\label{2.1}
\l(\sum_ia_i\r)^s
\le s
\sum_ia_i\l(\sum_{j\ge i}a_j\r)^{s-1},
\end{equation}
where $(a_i)_{i\in\Z}$ is a sequence of summable nonnegative reals. 
First, we verify the simple case 
$1<s\le 2$.
It follows from \eqref{2.1} that
\begin{alignat*}{2}
\lefteqn{
\int_{\Omega}\l(\sum_i\cE_i(\al_iw)\r)^s\,d\mu
}\\ &\le s
\sum_i
\int_{\Omega}
\cE_i(\al_iw)
\l(\sum_{j\ge i}\cE_j(\al_jw)\r)^{s-1}
\,d\mu
\\ &= s
\sum_i
\int_{\Omega}
\cE_i\cE_i(\al_iw)
\l(\sum_{j\ge i}\cE_j(\al_jw)\r)^{s-1}
\,d\mu
\\ &= s
\sum_i
\int_{\Omega}
\cE_i(\al_iw)
\cE_i\l[\l(\sum_{j\ge i}\cE_j(\al_jw)\r)^{s-1}\r]
\,d\mu,
\end{alignat*}
where we have used the fact that conditional expectation operators are selfadjoint. 
We notice that $s-1\le 1$. 
From Jensen's inequality and the tower rule of conditional expectations,
\begin{alignat*}{2}
&\le s
\sum_i
\int_{\Omega}
\cE_i(\al_iw)
\l(\sum_{j\ge i}\cE_i\cE_j(\al_jw)\r)^{s-1}
\,d\mu
\\ &= s
\int_{\Omega}
\sum_i
\cE_i(\al_iw)
\l(\cE_i(\ol{\al}_iw)\r)^{s-1}
\,d\mu.
\end{alignat*}
Next, we prove the case $s>2$. 
Let $k=\lceil s-2 \rceil$ 
be the smallest integer greater than $s-2$. 
Applying \eqref{2.1} $(k+1)$-times, we have 
\begin{alignat*}{2}
A_1
&=
\int_{\Omega}\l(\sum_i\cE_i(\al_iw)\r)^s\,d\mu
\\ &\le s(s-1)\cdots(s-k)
\\ &\times
\sum_{i_k\ge\cdots\ge i_1\ge i_0}
\int_{\Omega}
\cE_{i_0}(\al_{i_0}w)
\cE_{i_1}(\al_{i_1}w)
\cdots
\cE_{i_k}(\al_{i_k}w)
\l(\sum_{j\ge i_k}\cE_j(\al_jw)\r)^{s-k-1}
\,d\mu.
\end{alignat*}
Since 
$\cE_{i_0}(\al_{i_0}w)
\cE_{i_1}(\al_{i_1}w)
\cdots
\cE_{i_k}(\al_{i_k}w)$
becomes an $\cF_{i_k}$-measurable function,
the integral of the right-hand sides is equals to 
\begin{alignat*}{2}
\lefteqn{
\int_{\Omega}
\cE_{i_k}
\l[
\cE_{i_0}(\al_{i_0}w)
\cE_{i_1}(\al_{i_1}w)
\cdots
\cE_{i_k}(\al_{i_k}w)
\r]
\l(\sum_{j\ge i_k}\cE_j(\al_jw)\r)^{s-k-1}
\,d\mu
}\\ &=
\int_{\Omega}
\cE_{i_0}(\al_{i_0}w)
\cE_{i_1}(\al_{i_1}w)
\cdots
\cE_{i_k}(\al_{i_k}w)
\cE_{i_k}
\l[\l(\sum_{j\ge i_k}\cE_j(\al_jw)\r)^{s-k-1}\r]
\,d\mu
\\ &\le
\int_{\Omega}
\cE_{i_0}(\al_{i_0}w)
\cE_{i_1}(\al_{i_1}w)
\cdots
\cE_{i_k}(\al_{i_k}w)
\l(\cE_{i_k}(\ol{\al}_{i_k}w)\r)^{s-k-1}
\,d\mu,
\end{alignat*}
where we have used $0<s-k-1\le 1$. 
This yields 
$$
A_1
\le C
\int_{\Omega}
\l(\sum_i\cE_i(\al_iw)\r)^k
\l(\sum_i\cE_i(\al_iw)\l(\cE_i(\ol{\al}_iw)\r)^{s-k-1}\r)
\,d\mu.
$$
H\"{o}lder's inequality with exponent 
$\ds\frac{k}{s-1}+\frac{s-k-1}{s-1}=1$
gives 
$$
\sum_i\cE_i(\al_iw)\l(\cE_i(\ol{\al}_iw)\r)^{s-k-1}
\le
\l(\sum_i\cE_i(\al_iw)\r)^{\frac{k}{s-1}}
\l(\sum_i\cE_i(\al_iw)\l(\cE_i(\ol{\al}_iw)\r)^{s-1}\r)^{\frac{s-k-1}{s-1}},
$$
and, hence, 
$$
A_1
\le C
\int_{\Omega}
\l(\sum_i\cE_i(\al_iw)\r)^{\frac{sk}{s-1}}
\l(\sum_i\cE_i(\al_iw)\l(\cE_i(\ol{\al}_iw)\r)^{s-1}\r)^{\frac{s-k-1}{s-1}}
\,d\mu.
$$
H\"{o}lder's inequality with the same exponent gives 
$$
A_1
\le C
\l\{\int_{\Omega}\l(\sum_i\cE_i(\al_iw)\r)^s\,d\mu\r\}^{\frac{k}{s-1}}
\l\{\int_{\Omega}\sum_i\cE_i(\al_iw)\l(\cE_i(\ol{\al}_iw)\r)^{s-1}\,d\mu\r\}^{\frac{s-k-1}{s-1}}.
$$
Thus, we obtain 
$$
A_1\le CA_1^{\frac{k}{s-1}}A_2^{\frac{s-k-1}{s-1}}.
$$
This implies $A_1\le CA_2$.

\noindent{\bf (ii)}\ \ 
We prove $A_2\le CA_3$. 
It follows that 
\begin{alignat*}{2}
A_2
&=
\int_{\Omega}\sum_i\al_i\cE_iw\l(\cE_i(\ol{\al}_iw)\r)^{s-1}\,d\mu
\\ &\le 
\int_{\Omega}
\l\{\sum_i\al_i\cE_iw\r\}
\l\{\sup_j\cE_j(\ol{\al}_jw)\r\}^{s-1}
\,d\mu.
\end{alignat*}
H\"{o}lder's inequality gives 
$$
A_2
\le
A_1^{\frac1s}
A_3^{\frac1{s'}}.
$$
Since we have had $A_1\le CA_2$, 
we obtain 
$$
A_2
\le C
A_2^{\frac1s}
A_3^{\frac1{s'}},
$$
and $A_2\le CA_3$. 

\noindent{\bf (iii)}\ \ 
We prove $A_3\le CA_1$. 
It follows that 
\begin{alignat*}{2}
A_3
&=
\int_{\Omega}\l(\sup_i\cE_i(\ol{\al}_iw)\r)^s\,d\mu
\\ &\le 
\int_{\Omega}
\l\{\sup_i\cE_i\l(\sum_j\al_jw\r)\r\}^s\,d\mu
\\ &\le C
\int_{\Omega}\l(\sum_i\al_iw\r)^s\,d\mu
\\ &= CA_1,
\end{alignat*}
where we have used $s>1$ and Doob's maximal inequality. 
This completes the proof. 
\end{proof}

\subsection{Proof of Theorem 1.1}\label{ssec2.2}
Without loss of generality we may assume that 
$f$ is a nonnegative function. 
By duality (a) is equivalent to 
\begin{equation}\label{2.2}
\|T_{\al}(gw)\|_{L^{p'}(d\mu)}
\le C
\|g\|_{L^{q'}(wd\mu)}.
\end{equation}
It follows from Lemma \ref{lem2.3} that 
\begin{alignat*}{2}
\|T_{\al}(gw)\|_{L^{p'}(d\mu)}^{p'}
&\approx
\int_{\Omega}
\sum_i\al_i\cE_i(gw)\l(\cE_i(\ol{\al}_igw)\r)^{p'-1}
\,d\mu
\\ &\approx
\int_{\Omega}
\sum_i\al_i\ol{\al}_i^{p'-1}\l(\cE_i(gw)\r)^{p'}
\,d\mu,
\end{alignat*}
where we have used the condition \eqref{1.5}. 

We denote $\cE^w_if$ by 
the conditional expectation of $f$ with respect to $\cF_i$, 
$wd\mu$ in place of $d\mu$. 
We now claim that there holds 
$$
\frac{\cE_i(gw)}{\cE_iw}
=
\cE^w_ig.
$$
Indeed, by a simple limiting argument, 
if necessary, we can assume that $g$ is a bounded function. 
Then, since 
$\ds\frac{\cE_i(gw)}{\cE_iw}$ 
is an $\cF_i$-measurable function and 
belongs to $\cL^{+}$, 
for any $E\in\cF_i$,
\begin{alignat*}{2}
\lefteqn{
\int_{E}\frac{\cE_i(gw)}{\cE_iw}w\,d\mu
=
\int_{E}
\cE_i\l(\frac{\cE_i(gw)}{\cE_iw}\r)w\,d\mu
}\\ &=
\int_{E}
\frac{\cE_i(gw)}{\cE_iw}\cE_iw\,d\mu
=
\int_{E}\cE_i(gw)\,d\mu
=
\int_{E}gw\,d\mu.
\end{alignat*}

This claim yields 
$$
\|T_{\al}(gw)\|_{L^{p'}(d\mu)}^{p'}
\approx
\int_{\Omega}\sum_ia_i(\cE_iw)^{p'}(\cE^w_ig)^{p'}\,d\mu,
$$
where $a_i:=\al_i\ol{\al}_i^{p'-1}$. 
Thus, \eqref{2.2} is equivalent to 
\begin{equation}\label{2.3}
\l(\int_{\Omega}\sum_ia_i(\cE_iw)^{p'}(\cE^w_ig)^{p'}\,d\mu\r)^{\frac1{p'}}
\le C
\|g\|_{L^{q'}(wd\mu)}.
\end{equation}
By the Carleson embedding theorem (Corollary \ref{cor3.4} below), 
\eqref{2.3} is equivalent to the statement that 
there exists a constant $C>0$ such that 
\begin{equation}\label{2.4}
\int_{E}\sum_{j\ge i}a_j(\cE_jw)^{p'}\,d\mu
\le C
[wd\mu](E)^{\frac{p'}{q'}}
\end{equation}
holds for any $E\in\cF_i^0$, 
$i\in\Z$. 
{}From the condition \eqref{1.5} and Lemma 
\ref{lem2.3}\footnote{
We let 
$\ds
\al_j
:=
\begin{cases}
0\text{ for }j<i,
\\
1_{E}\al_j\text{ for }j\ge i.
\end{cases}
$}
there holds 
\begin{alignat*}{2}
\lefteqn{
\int_{E}\sum_{j\ge i}a_j(\cE_jw)^{p'}\,d\mu
}\\ &=
\int_{E}\sum_{j\ge i}\al_j\ol{\al}_j^{p'-1}(\cE_jw)^{p'}\,d\mu
\\ &\approx
\int_{E}
\sum_{j\ge i}
\al_j\cE_jw
\l(\cE_j(\ol{\al}_jw)\r)^{p'-1}
\,d\mu
\\ &\approx
\int_{E}\l(\sum_{j\ge i}\al_j\cE_jw\r)^{p'}\,d\mu.
\end{alignat*}
Hence, \eqref{2.4} is equivalent to 
$$
\l(\int_{E}\l(\sum_{j\ge i}\al_j\cE_jw\r)^{p'}\,d\mu\r)^{\frac1{p'}}
\le C
[wd\mu](E)^{\frac1{q'}}.
$$
Then we finish the proof. $\square$

\subsection{Proof of Theorem 1.2}\label{ssec2.3}
We need another  lemma. 

\begin{lemma}\label{lem2.4}
Let $1<p<\infty$, 
$\al$ satisfy the condition {\rm\eqref{1.5}} 
and $w$ be a weight. Then 
$$
\|T_{\al}f\|_{L^p(vd\mu)}
\le C
\|f\|_{L^p(d\mu)},
$$
where 
$$
v:=\frac{w}{\cW_{\al}[w]^{p-1}}
\text{ and }
\cW_{\al}[w]
:=
\sum_i\al_i\ol{\al}_i^{p'-1}(\cE_iw)^{p'-1}.
$$
\end{lemma}

\begin{proof}
We need only verify that 
the weight $v$ fulfill \eqref{2.4} with $q=p$. 
It suffices to show that 
there exists a constant $C>0$ such that 
$$
\int_{E}\sum_{j\ge i}a_j(\cE_jv)^{p'}\,d\mu
\le C
[vd\mu](E)
$$
holds for any $E\in\cF_i^0$, 
$i\in\Z$, where 
$a_j=\al_j\ol{\al}_j^{p'-1}$. 

By conditional H\"{o}lder's inequality we see that 
$$
(\cE_jv)^{p'}
\le
(\cE_jw)^{p'-1}
\cE_j\l(\frac{w}{\cW_{\al}[w]^p}\r).
$$
This implies 
\begin{alignat*}{2}
\lefteqn{
\int_{E}\sum_{j\ge i}a_j(\cE_jv)^{p'}\,d\mu
}\\ &\le
\int_{E}
\sum_{j\ge i}
a_j(\cE_jw)^{p'-1}
\cE_j\l(\frac{w}{\cW_{\al}[w]^p}\r)
\,d\mu
\\ &\approx
\int_{E}
\sum_{j\ge i}
\cE_j\l[a_j(\cE_jw)^{p'-1}\frac{w}{\cW_{\al}[w]^p}\r]
\,d\mu
\\ &\approx
\int_{E}
\sum_{j\ge i}
a_j(\cE_jw)^{p'-1}\frac{w}{\cW_{\al}[w]^p}
\,d\mu
\\ &\le C
\int_{E}\frac{w}{\cW_{\al}[w]^{p-1}}\,d\mu
=C
[vd\mu](E),
\end{alignat*}
where we have used the condition \eqref{1.5}. 
This is our desired inequality. 
\end{proof}

\begin{proof}[Proof of (b) $\,\Rightarrow\,$ (a)]
Recall that in this case 
$0<q<p<\infty$, $1<p<\infty$ and 
$\ds\frac1r=\frac1q-\frac1p$.
H\"{o}lder's inequality gives 
\begin{alignat*}{2}
\|T_{\al}f\|_{L^q(wd\mu)}
&=
\|\cW_{\al}[w]^{\frac1{p'}}\cdot \cW_{\al}[w]^{-\frac1{p'}}T_{\al}f\|_{L^q(wd\mu)}
\\ &\le
\|\cW_{\al}[w]^{\frac1{p'}}\|_{L^r(wd\mu)}
\|T_{\al}f(\cW_{\al}[w])^{-\frac1{p'}}\|_{L^p(wd\mu)}
\\ &\le C
\|\cW_{\al}[w]^{\frac1{p'}}\|_{L^r(wd\mu)}
\|f\|_{L^p(wd\mu)},
\end{alignat*}
where in the last inequality we have used Lemma \ref{lem2.4}. 
\end{proof}

\begin{proof}[Proof of (a) $\,\Rightarrow\,$ (b)]
Recall now that $1<q<\infty$ and 
$\ds\frac1r=\frac1q-\frac1p$. 
Our standing assumption is that \eqref{2.3} holds. 
Then the statement (b) is also a consequence of the Carleson embedding theorem (Corollary \ref{cor3.6} below). 
\end{proof}

\section{Carleson embedding theorem}\label{sec3}
In this section 
we will discuss the well-known Carleson embedding theorem in a filtered measure space.
The Carleson embedding theorem proved here is a refinement of several previous results
which are generalizations of the classical Carleson embedding theorem.
The related works we would like to mention are
\cite{Ke,Ar,Lo,BlJa,Gra}.

Kemppainen \cite{Ke} treats the Carleson embedding theorem 
in $\sg$-finite filtered measure space. 
His result corresponds to the case $p=2$ 
of Corollary \ref{cor3.4} below. 
Although his argument can be adapted to our situation, 
our assumptions about a filtered measure space is weaker than his and 
we also treat not only weighted measure 
but also general measure for the Carleson measure. 
Treating $p\ne q$ case is also new compared with his result. 
Related results which treats a vector-valued case are 
in \cite{Hy1,HyMcPo,HyKe}. 

Arai \cite{Ar} treats the Carleson measure 
in a continuously filtered probability space.  
His result corresponds to Corollaries \ref{cor3.3} and \ref{cor3.4}. 
While he only treats a probability space, 
we treat a $\sg$-finite measure space. 
Notice also that his Carleson measure definition uses any stopping time whereas our definition uses only a special type of stopping times. 

Long \cite{Lo} treats the Carleson measure in a discretely filtered probability space and proves the Carleson embedding theorem.
His result can be regarded as the discrete version of the result in Arai \cite{Ar}.
He treats all $(\cF_i)_{i\in\Z}$-measurable functions in his formulation of the Carleson embedding theorem 
and Theorem \ref{thm3.1} is similar to it in this respect.
  
The work of Blasco and Jarchows \cite{BlJa} 
investigates the Carleson measure on 
$\bar{D}$, i.e., 
a finite positive Borel measure $\mu$ on $ \bar{D} $
such that, for given values of 
$0<p,q<\infty$, 
the embedding 
$J_{\mu}:\,H^p(D)\rightarrow L^q(\bar{D},\mu)$ 
exists. Here, $D$ is a unit ball in the plane and $\bar{D}$ is a closure of the unit ball. 
Our theorems in this section which treat different exponents $ p, q $ are 
the analogues of their results in a setting of a filtered measure space.

Theorem \ref{thm3.1} corresponds to 
\cite[Theorem 7.3.5.]{Gra} 
which is the Carleson embedding theorem 
for functions in the half space. 
Theorem \ref{thm3.1} 
(,resp., \cite[Theorem 7.3.5.]{Gra}) 
treats arbitrary measurable functions instead of martingales 
(,resp., harmonic functions). 

Throughout this section 
we let 
$(\Omega,\cF,\mu;(\cF_i)_{i\in\Z})$ 
be a $\sg$-finite filtered measure space.
We also let $f_i$, $i\in\Z$, be an $\cF_i$-measurable nonnegative real-valued function 
and $\nu_i$ be a measure on $\cF_i$. 
Set a maximal function of $f=(f_i)$ by 
$f^{*}:=\sup_if_i$. 

\begin{theorem}\label{thm3.1}
Let $\theta\ge 1$ be arbitrarily taken and be fixed. 
Then the following conditions are equivalent:

\begin{enumerate}
\item[{\rm(i)}] 
There exists a constant $C_0>0$ such that 
$$
\sum_{j\ge i}\nu_j(E)\le C_0\mu(E)^{\theta}
$$
for any $E\in\cF_i$, $i\in\Z;$
\item[{\rm(ii)}] 
For \lq\lq any" $p\in(0,\infty)$ 
there exists a constant $C_p>0$ such that 
$$
\l(\sum_{i\in\Z}\int_{\Omega}f_i^{p\theta}\,d\nu_i\r)^{\frac1{p\theta}}
\le C_p
\|f^{*}\|_{L^p(d\mu)};
$$
\item[{\rm(iii)}] 
For \lq\lq some" $p_0\in(0,\infty)$ 
there exists a constant $C_{p_0}>0$ such that 
$$
\l(\sum_{i\in\Z}\int_{\Omega}f_i^{p_0\theta}\,d\nu_i\r)^{\frac1{p_0\theta}}
\le C_{p_0}
\|f^{*}\|_{L^{p_0}(d\mu)}.
$$
\end{enumerate}
Moreover,
the least possible $C_0$ and $C_p$ enjoy 
$$
C_p\le (C_0\theta)^{\frac1{p\theta}},
\qquad
C_0\le C_p^{p\theta}.
$$
\end{theorem}

\begin{proof}
By a standard limiting argument, 
we can replace $ \Z $ by $\N$. 
Hence we consider 
$f_i$, $i\in\N$, and $\nu_i$, $i\in\N$. 

\noindent{\bf (i)~$\Rightarrow$~(ii)}\ \ 
For $\la>0$ we set 
$F=\{f^{*}>\la\}$ and 
$F_i=\{f_i>\la\}$, $i\in\N$. 
Then we have $F=\ds\bigcup_iF_i$. 
We define a stopping time $\tau$ by 
$$
\tau:=\inf\{i:\,f_i>\la\}.
$$
Using this, we set
$$
G_i=\{\tau=i\}
$$
for $i\in\N$.
Then we easily see that $G_i$'s are disjoint and that 
$F_i\subset\ds\bigcup_{j=0}^iG_j$. 
Hence, we have 
$F=\ds\bigcup_{i\in\N}G_i$. 

We define a measure space 
$(\Omega\times\N,\cG,\nu)$ by the following: 
\begin{enumerate}
\item[{\rm(1)}] 
$\cG$ is generated by 
$(\{i\}\times\cF_i)_{i\in\N}$;
\item[{\rm(2)}] 
$\nu|_{\{i\}\times\cF_i}=\nu_i$.
\end{enumerate}
\noindent
We can easily see that there exists a unique measure $\nu$ on $\cG$ satisfying \rm{(2)}.
We regard $f=(f_i)$ as a function on $\Omega\times\N$. 
Then we see that $f$ is a $\cG$-measurable function on $\Omega\times\N$. 

We estimate $\nu(\{f>\la\})$ from above by 
$\mu(\{f^{*}>\la\})$ as follows:
\begin{alignat*}{2}
\nu(\{f>\la\})
&\le
\sum_i\nu_i(F_i)
\le
\sum_i\nu_i\l(\bigcup_{0\le j\le i}G_j\r) 
=
\sum_j\sum_{i\ge j}\nu_i(G_j)
\le C_0
\sum_j\mu(G_j)^{\theta}
\\ &\le C_0
\l(\sum_j\mu(G_j)\r)^{\theta}
\le C_0
\mu\l(\bigcup_jG_j\r)^{\theta}
\le C_0
\mu(F)^{\theta}
=C_0
\mu(\{f^{*}>\la\})^{\theta},
\end{alignat*}
where we have used the assertion (i) and the fact that $\theta\ge 1$. 
Thus, we obtain 
\begin{alignat*}{2}
\lefteqn{
\frac1{p\theta}
\int_{\Omega\times\N}f^{p\theta}\,d\nu
=
\int_{0}^{\infty}
\la^{p\theta -1}\nu(\{f>\la\})
\,d\la
}\\ &\le C_0
\int_{0}^{\infty}
\la^{p\theta-1}\mu(\{f^{*}>\la\})^{\theta}
=C_0
\int_{0}^{\infty}
\l(\la^p\mu(\{f^{*}>\la\})\r)^{\theta-1}
\l(\la^{p-1}\mu(\{f^{*}>\la\})\r)
\,d\la
\\ &\le \frac{C_0}{p}
\|f^{*}\|_{L^p(d\mu)}^{p\theta-p}
\cdot
p\int_{0}^{\infty}\la^{p-1}\mu(\{f^{*}>\la\})\,d\la
=\frac{C_0}{p}
\|f^{*}\|_{L^p(d\mu)}^{p\theta-p}\|f^{*}\|_{L^p(d\mu)}^p
\\ &=\frac{C_0}{p}
\|f^{*}\|_{L^p(d\mu)}^{p\theta},
\end{alignat*}
where we have used Chebyshev's inequality. 
Taking $\ds\frac1{p\theta}$th power in both sides, we obtain 
$$
\l(\sum_i\int_{\Omega}f_i^{p\theta}\,d\nu_i\r)^{\frac1{p\theta}}
\le (C_0\theta)^{\frac1{p\theta}}
\|f^{*}\|_{L^p(d\mu)}.
$$
Hence we obtain the assertion (ii). 

\noindent{\bf (ii)~$\Rightarrow$~(iii)}\ \ 
Obvious.

\noindent{\bf (iii)~$\Rightarrow$~(i)}\ \ 
It suffices to take 
$\ds
f_j
:=
\begin{cases}
0\text{ for }j<i,
\\
1_{E}\text{ for }j\ge i.
\end{cases}
$. \\
This completes the proof.
\end{proof}

\begin{remark}\label{rem3.2}
Let $ \theta \geq 1. $
Let a sigma-algebra $ \cG $ on $ \Omega\times\Z $ be 
generated by $ (\{i\} \times \cF_i)_{i \in \Z}. $
We call a measure $\nu$ 
which is defined on $(\Omega\times\Z,\cG)$ 
a $\theta$-Carleson measure on $\Omega\times\Z$
if $ \nu_i := \nu|_{\{i\} \times \cF_i} $, $ i \in \Z $ satisfy the condition
(i) in Theorem \ref{thm3.1}.  
We call the infimum of $C_0$ in (i) in Theorem \ref{thm3.1} 
the $\theta$-Carleson measure norm of $ \nu $.
It is easy to see that the condition (i) in Theorem \ref{thm3.1} is equivalent to 
\begin{alignat}{1}\label{3.1}
\sup_{\tau}
\mu(\{\tau<\infty\})^{-\theta}
\nu(\{(\omega,k)\in\Omega\times\Z:\,k\ge \tau(\omega)\})
<\infty,
\end{alignat} 
where $\tau$ runs through all stopping times where $\mu(\{\tau<\infty\})$ is nonzero and finite,
and that $\theta$-Carleson measure norm of $\nu$ is equal to this quantity.
The concept of a "$\theta$-Carleson measure" 
was first introduced by \cite{Ar,Lo}
using \eqref{3.1} as a definition
when $ \theta = 1. $ 
\end{remark}

Thanks to Doob's maximal inequality, 
we have the following corollary of the theorem.

\begin{corollary}\label{cor3.3}
Let $(f_i)_{i\in\Z}$ be a martingale on $(\Omega,\cF,\mu)$ 
and $\theta\ge 1$ be arbitrarily taken and be fixed.
Then the following conditions are equivalent:

\begin{enumerate}
\item[{\rm(i)}] 
There exists a constant $C_0>0$ such that 
$$
\sum_{j\ge i}\nu_j(E)\le C_0\mu(E)^{\theta}
$$
for any $E\in\cF_i$, $i\in\Z;$
\item[{\rm(ii)}] 
For \lq\lq any" $p\in(1,\infty)$ 
there exists a constant $C_p>0$ such that 
$$
\l(\sum_{i\in\Z}\int_{\Omega}|f_i|^{p\theta}\,d\nu_i\r)^{\frac1{p\theta}}
\le C_p
\sup_{i\in\Z}\|f_i\|_{L^p(d\mu)};
$$
\item[{\rm(iii)}] 
For \lq\lq some" $p_0\in(1,\infty)$ 
there exists a constant $C_{p_0}>0$ such that 
$$
\l(\sum_{i\in\Z}\int_{\Omega}|f_i|^{p_0\theta}\,d\nu_i\r)^{\frac1{p_0\theta}}
\le C_{p_0}
\sup_{i\in\Z}\|f_i\|_{L^{p_0}(d\mu)}.
$$
\end{enumerate}
Moreover,
the least possible $C_0$ and $C_p$ enjoy 
$$
C_p\le C(C_0\theta)^{\frac1{p\theta}},
\qquad
C_0\le C_p^{p\theta}.
$$
\end{corollary}

We have another corollary, where we only consider martingales consisting of the conditional expectations of a function.

\begin{corollary}\label{cor3.4}
Let $\theta\ge 1$ be arbitrarily taken and be fixed.
Then the following conditions are equivalent:

\begin{enumerate}
\item[{\rm(i)}] 
There exists a constant $C_0>0$ such that 
$$
\sum_{j\ge i}\nu_j(E)\le C_0\mu(E)^{\theta}
$$
for any $E\in\cF_i$, $i\in\Z;$
\item[{\rm(ii)}] 
For \lq\lq any" $p\in(1,\infty)$ 
there exists a constant $C_p>0$ such that 
$$
\l(\sum_{i\in\Z}\int_{\Omega}|\cE_if|^{p\theta}\,d\nu_i\r)^{\frac1{p\theta}}
\le C_p
\|f\|_{L^p(d\mu)};
$$
\item[{\rm(iii)}] 
For \lq\lq some" $p_0\in(1,\infty)$ 
there exists a constant $C_{p_0}>0$ such that 
$$
\l(\sum_{i\in\Z}\int_{\Omega}|\cE_if|^{p_0\theta}\,d\nu_i\r)^{\frac1{p_0\theta}}
\le C_{p_0}
\|f\|_{L^{p_0}(d\mu)}.
$$
\end{enumerate}
Moreover,
the least possible $C_0$ and $C_p$ enjoy 
$$
C_p\le C(C_0\theta)^{\frac1{p\theta}},
\qquad
C_0\le C_p^{p\theta}.
$$
\end{corollary}

We next consider the Carleson embedding theorem for the case $q<p$. 

\begin{theorem}\label{thm3.5}
Let $w_i$, $i\in\Z$, be an $\cF_i$-measurable nonnegative real-valued function 
and $w\in\cL^{+}$ be a weight. 
Suppose that $\ds\frac{w_i}{\cE_iw}$, 
$i\in\Z$, belong to the class $\cL^{+}$. 
Let $\theta>1$ be arbitrarily taken and be fixed.
Then the following conditions are equivalent:

\begin{enumerate}
\item[{\rm(i)}] 
There exists a constant $C_0>0$ such that 
$$
\l\|\sum_{i\in\Z}\frac{w_i}{\cE_iw}\r\|_{L^{\theta'}(wd\mu)} \leq C_0;
$$
\item[{\rm(ii)}] 
For \lq\lq any" $q\in(0,\infty)$ 
there exists a constant $C_q>0$ such that 
$$
\l(\sum_{i\in\Z}\int_{\Omega}w_if_i^q\,d\mu\r)^{\frac1q}
\le C_q
\|f^{*}\|_{L^{q\theta}(wd\mu)};
$$
\item[{\rm(iii)}] 
For \lq\lq some" $q_0\in(0,\infty)$ 
there exists a constant $C_{q_0}>0$ such that 
$$
\l(\sum_{i\in\Z}\int_{\Omega}w_if_i^{q_0}\,d\mu\r)^{\frac1{q_0}}
\le C_{q_0}
\|f^{*}\|_{L^{q_0\theta}(wd\mu)}.
$$
\end{enumerate}
Moreover,
the least possible $C_0$ and $C_q$ enjoy 
$$
C_q\le C_0^{\frac1q},
\qquad
C_0\le CC_q.
$$
\end{theorem}

\begin{proof}
\noindent{\bf (i)~$\Rightarrow$~(ii)}\ \ 
It follows that 
$$
\sum_i\int_{\Omega}w_if_i^q\,d\mu
=
\sum_i\int_{\Omega}\frac{w_i}{\cE_iw}f_i^q\cE_iw\,d\mu.
$$
By a simple limiting argument, 
if necessary, we can assume that $f_i$ is a bounded function. 
Then, since 
$\ds\frac{w_i}{\cE_iw}f_i^q$
is an $\cF_i$-measurable function and 
belongs to the class $\cL^{+}$, 
\begin{alignat*}{2}
&=
\sum_i\int_{\Omega}\cE_i\l[\frac{w_i}{\cE_iw}f_i^qw\r]\,d\mu
\\ &=
\sum_i\int_{\Omega}\frac{w_i}{\cE_iw}f_i^qw\,d\mu
\\ &\le
\int_{\Omega}
\l(\sum_i\frac{w_i}{\cE_iw}\r)
\l(\sup_jf_j\r)^q
w\,d\mu.
\end{alignat*}
H\"{o}lder's inequality with exponent $\theta$ gives 
$$
\le 
\l\|\sum_i\frac{w_i}{\cE_iw}\r\|_{L^{\theta'}(wd\mu)}
\|f^{*}\|_{L^{q\theta}(wd\mu)}^q.
$$
This yields the assertion (ii) with 
$\ds C_q\le C_0^{\frac1q}$.

\noindent{\bf (ii)~$\Rightarrow$~(iii)}\ \ 
Obvious. 

\noindent{\bf (iii)~$\Rightarrow$~(i)}\ \ 
It follows that, for nonnegative 
$g\in L^{\theta}(wd\mu)\cap L^{\infty}(wd\mu)$, 
\begin{alignat*}{2}
\lefteqn{
\int_{\Omega}\l(\sum_i\frac{w_i}{\cE_iw}\r)gw\,d\mu
=
\sum_i\int_{\Omega}\frac{w_i}{\cE_iw}gw\,d\mu
}\\ &=
\sum_i\int_{\Omega}\cE_i\l(\frac{w_i}{\cE_iw}\r)gw\,d\mu
=
\sum_i\int_{\Omega}\frac{w_i}{\cE_iw}\cE_i(gw)\,d\mu
\\ &=
\sum_i\int_{\Omega}w_i\frac{\cE_i(gw)}{\cE_iw}\,d\mu
=
\sum_i\int_{\Omega}w_i\l\{\l(\frac{\cE_i(gw)}{\cE_iw}\r)^{\frac1{q_0}}\r\}^{q_0}\,d\mu
\\ &\le C_{q_0}
\l\|\l(\sup_i\frac{\cE_i(gw)}{\cE_iw}\r)^{\frac1{q_0}}\r\|_{L^{q_0\theta}(wd\mu)}^{q_0}
\le CC_{q_0}
\|g\|_{L^{\theta}(wd\mu)},
\end{alignat*}
where we have used the assertion (iii) and 
Doob's maximal inequality.
By a limiting argument and duality we must have
$$
\l\|\sum_i\frac{w_i}{\cE_iw}\r\|_{L^{\theta'}(wd\mu)}
\le CC_{q_0}.
$$
This yields $C_0\le CC_{q_0}$ 
and completes the proof.
\end{proof}

\begin{corollary}\label{cor3.6}
Let $w_i$, $i\in\Z$, be an $\cF_i$-measurable nonnegative real-valued function 
and $w\in\cL^{+}$ be a weight. 
Suppose that $\ds\frac{w_i}{\cE_iw}$, 
$i\in\Z$, belong to the class $\cL^{+}$. 
Let $0<q<p<\infty$ and 
$1<p<\infty$. Then 
the following conditions are equivalent:

\begin{enumerate}
\item[{\rm(i)}] 
For $\ds\frac1r=\frac1q-\frac1p$ 
there exists a constant $C_1>0$ such that 
$$
\l\|\l(\sum_{i\in\Z}\frac{w_i}{\cE_iw}\r)^{\frac1q}\r\|_{L^r(wd\mu)} \leq C_1;
$$
\item[{\rm(ii)}] 
There exists a constant $C_2>0$ such that 
$$
\l(\sum_{i\in\Z}\int_{\Omega}w_i|\cE_if|^q\,d\mu\r)^{\frac1q}
\le C_2
\|f\|_{L^p(wd\mu)}.
$$
\end{enumerate}
Moreover,
the least possible $C_1$ and $C_2$ are equivalent. 
\end{corollary}

\begin{proof}
It suffices to notice Doob's maximal inequality 
and fact that 
$$
\frac1{(p/q)'}\frac1q
=
\l(1-\frac{q}{p}\r)\frac1q
=
\frac{p-q}{pq}
=\frac1r.
$$
\end{proof}

\section{Two-weight norm inequalities for generalized Doob's maximal operator}\label{sec4}
In this section, 
by the use of the Carleson embedding theorem, 
we give a simple proof of the analogue of Sawyer's theorem \cite{Sa1} 
characterizing the weights governing the 
two-weight strong-type norm inequality 
for generalized Doob's maximal operator $M_{\al}$. 
The following is the analogue of Sawyer's theorem in a martingale setting. 

\begin{theorem}\label{thm4.1}
Let $1<p\le q<\infty$, 
$\al_i$, $i\in\Z$, be a nonnegative bounded $\cF_i$-measurable function and 
$u,v\in\cL^{+}$ be a weight. Then 
the following statements are equivalent:

\begin{enumerate}
\item[{\rm(a)}] 
There exists a constant $C_1>0$ such that 
$$
\|M_{\al}f\|_{L^q(ud\mu)}
\le C_1
\|f\|_{L^p(vd\mu)};
$$
\item[{\rm(b)}] 
If $\sg=v^{1-p'}\in\cL^{+}$, then 
there exists a constant $C_2>0$ such that 
$$
\l(\int_{E}\l(\sup_{j\ge i}\al_j\cE_j\sg\r)^qu\,d\mu\r)^{\frac1q}
\le C_2
[\sg d\mu](E)^{\frac1p}
$$
for any $E\in\cF_i^0$, $i\in\Z.$
\end{enumerate}
Moreover,
the least possible $C_1$ and $C_2$ are equivalent.
\end{theorem}

\begin{proof}
We follow the argument in \cite{Cr}. 

The proof of 
(a)~$\Rightarrow$~(b) 
follows at once if we substitute the test function 
$f=1_{E}\sg$. 
We shall prove converse. 

Without loss of generality we may assume that 
$f$ is a nonnegative function. 
For $j\in\Z$ define a stopping time 
$$
\tau_j:=\inf\{i:\,\al_i\cE_if>2^j\}.
$$
Clearly, $\tau_j\le\tau_{j+1}$. 
If we let 
$$
F_j
:=
\{-\infty<\tau_j<\infty\}
\setminus
\{-\infty<\tau_{j+1}<\infty\},
$$
then we see that 
$F_j$'s are disjoint and 
$$
\{M_{\al}f>0\}
=
\bigcup_jF_j.
$$
We now set 
$$
E_j^i
:=
F_j\cap\{\tau_j=i\}.
$$
It follows that 
$E_j^i$'s are disjoint, 
$F_j=\ds\bigcup_iE_j^i$ 
and, if $E_j^i\ne\emptyset$, then 
$$
M_{\al}f
\approx
\al_i\cE_if
\text{ on }
E_j^i.
$$
We now estimate as follows:
\begin{alignat*}{2}
\int_{\Omega}(M_{\al}f)^qu\,d\mu
&=
\sum_{i,j}
\int_{E_j^i}(M_{\al}f)^qu\,d\mu
\le C
\sum_{i,j}
\int_{E_j^i}(\al_i\cE_if)^qu\,d\mu
\\ =C
\sum_i
\int_{\Omega}
\l(\sum_j1_{E_j^i}(\al_i\cE_i\sg)^q\r)
\l(\frac{\cE_if}{\cE_i\sg}\r)^q
u\,d\mu.
\end{alignat*}
Since 
$\ds
\frac{\cE_if}{\cE_i\sg}
=
\cE^{\sg}_i\l[\frac{f}{\sg}\r]
$,
we shall evaluate 
$$
\sum_i
\int_{\Omega}
\l(\sum_j1_{E_j^i}(\al_i\cE_i\sg)^q\r)
\l(\cE^{\sg}_i\l[\frac{f}{\sg}\r]\r)^q
u\,d\mu.
$$
Applying the Carleson embedding theorem 
(Corollary \ref{cor3.4}), 
we need only verify that 
there exists a constant $C>0$ such that 
$$
\sum_{j\ge i}
\int_{E}
\l(\sum_k1_{E_k^j}(\al_j\cE_j\sg)^q\r)
u\,d\mu
\le C
[\sg d\mu](E)^{\frac{q}{p}}
$$
holds for any $E\in\cF_i$, $i\in\Z.$
The fact that $E_k^j$'s are disjoint 
and the assertion (b) yield 
$$
\int_{E}
\sum_{j\ge i}\sum_k
1_{E_k^j}(\al_j\cE_j\sg)^qu\,d\mu
\le
\int_{E}\l(\sup_{j\ge i}\al_j\cE_j\sg\r)^qu\,d\mu
\le C
[\sg d\mu](E)^{\frac{q}{p}}.
$$
This completes the proof. 
\end{proof}

The following lemma was proved in \cite[Theorem 1]{Ch}. 
For the sake of the completeness 
the full proof is given here.

\begin{lemma}\label{lem4.2}
Let $1<p<\infty$, 
$w\in\cL^{+}$ be a weight and 
$\sg=w^{1-p'}\in\cL^{+}$. Then 
the following statements are equivalent:

\begin{enumerate}
\item[{\rm(a)}] 
There exists a constant $C_1>0$ such that 
$$
\sup_{i\in\Z}
\|(\cE_iw)(\cE_i\sg)^{p-1}\|_{L^{\infty}(d\mu)}<C_1;
$$
\item[{\rm(b)}] 
There exists a constant $C_2>0$ such that 
$$
\int_{E}\l(\sup_{j\ge i}\cE_j\sg\r)^pw\,d\mu
\le C_2^p
[\sg d\mu](E)
$$
for any $E\in\cF_i^0$, $i\in\Z.$ 
\end{enumerate}
Moreover,
the least possible $C_1$ and $C_2$ enjoy 
$$
C_1\le C_2^p,
\qquad
C_2\le CC_1^{\frac1{p-1}}.
$$
\end{lemma}

\begin{proof}
Proof of (b)~$\Rightarrow$~(a). 
It follows that, 
for any $E\in\cF_i^0$, $i\in\Z$, 
\begin{alignat*}{2}
\int_{E}(\cE_iw)(\cE_i\sg)^p\,d\mu
&=
\int_{E}\cE[(\cE_i\sg)^pw]\,d\mu
=
\int_{E}(\cE_i\sg)^pw\,d\mu
\le
\int_{E}
\l(\sup_{j\ge i}\cE_j\sg\r)^p
w\,d\mu
\\ &\le C_2^p
\int_{E}\sg\,d\mu
=C_2^p
\int_{E}\cE_i\sg\,d\mu.
\end{alignat*}
This implies 
$$
(\cE_iw)(\cE_i\sg)^p
\le C_2^p
\cE_i\sg
$$
and, hence, yields (a) with 
$C_1\le C_2^p$. 

We now verify converse 
(a)~$\Rightarrow$~(b). 
By the assertion (a) we have 
$$
(\cE_i\sg)^p
\le C_1^{p'}
(\cE_iw)^{-p'}
=C_1^{p'}
\l(\cE^w_i[w^{-1}]\r)^{p'}.
$$
This yields, 
for any $E\in\cF_i^0$, $i\in\Z$, 
\begin{alignat*}{2}
\int_{E}\l(\sup_{j\ge i}\cE_j\sg\r)^pw\,d\mu
&\le C_1^{p'}
\int_{E}1_{E}\l(\sup_{j\ge i}\cE^w_j[w^{-1}]\r)^{p'}w\,d\mu
=
\int_{E}\l(\sup_{j\ge i}\cE^w_j[1_{E}w^{-1}]\r)^{p'}w\,d\mu
\\ &\le CC_1^{p'}
\int_{E}w^{1-p'}\,d\mu,
\end{alignat*}
where we have used Doob's maximal inequality.
Thus, we obtain 
$$
\int_{E}\l(\sup_{j\ge i}\cE_j\sg\r)^pw\,d\mu
\le
CC_1^{p'}[\sg d\mu](E)
$$
and have (b) with 
$\ds C_2=C_1^{\frac1{p-1}}$. 
This proves the theorem. 
\end{proof}

\begin{corollary}\label{cor4.3}
Let $1<p<\infty$, 
$\al_i$, $i\in\Z$, be a nonnegative bounded $\cF_i$-measurable function, 
$u,v\in\cL^{+}$ be a weight and 
$\sg=v^{1-p'}\in\cL^{+}$. Then, 
two-weight norm inequality 
$$
\|M_{\al}f\|_{L^p(ud\mu)}
\le C_1
\|f\|_{L^p(vd\mu)}
$$
holds \lq\lq if and only if" 
there exists a constant $C_2>0$ such that 
$$
\cE_i\l[\l(\sup_{j\ge i}\al_i\cE_j\sg\r)^pu\r]
\le
C_2^p\cE_i\sg.
$$
for any $i\in\Z.$
Moreover, 
the least possible $C_1$ and $C_2$ are equivalent.
\end{corollary}

\begin{remark}\label{rem4.4}
Long and Peng \cite{LoPe} 
showed that Corollary \ref{cor4.3} holds 
for Doob's maximal operator in a filtered probability space. 
(See also a recent work by Chen and Liu \cite{ChLi}.)
\end{remark}

\begin{corollary}\label{cor4.5}
Let $1<p<\infty$, 
$w\in\cL^{+}$ be a weight and 
$\sg=w^{1-p'}\in\cL^{+}$. Then, 
one-weight norm inequality 
$$
\|f^{*}\|_{L^p(wd\mu)}
\le C_1
\|f\|_{L^p(wd\mu)}
$$
holds \lq\lq if and only if" 
$$
\sup_{i\in\Z}
\|(\cE_iw)(\cE_i\sg)^{p-1}\|_{L^{\infty}(d\mu)}<C_2<\infty
$$
for any $i\in\Z.$ 
Moreover, 
the least possible $C_1$ and $C_2$ enjoy 
$$
C_2\le C_1^p,
\qquad
C_1\le CC_2^{\frac1{p-1}}.
$$
\end{corollary}

\begin{remark}\label{rem4.6}
For the $A_p$ weights with some regularity condition, 
Izumisawa and Kazamaki \cite{IzKa} 
proved first that Corollary \ref{cor4.5} holds in a filtered probability space.
Jawerth \cite{Ja} found that the added property is superfluous.
\end{remark}

\section{One-weight norm estimates of Hyt\"{o}nen-P\'{e}rez type for Doob's maximal operator}\label{sec5}
In this section, 
by an application of Theorem \ref{thm4.1},
we will sharpen Corollary \ref{cor4.5} 
following the argument due to Hyt\"{o}nen and P\'{e}rez (see \cite{Hy4,HyPe}). 

Let $1<p<\infty$, 
$w\in\cL^{+}$ be a weight and 
$\sg:=w^{1-p'}\in\cL^{+}$. 
We define 
$$
[w]_{A_p}
:=
\sup_{i\in\Z}
\|(\cE_iw)(\cE_i\sg)^{p-1}\|_{L^{\infty}(d\mu)}
$$
and define 
$$
[w]_{A_{\infty}}
:=
\sup_{i\in\Z}
\l\|(\cE_iw)\exp\l(-\cE_i(\log w)\r)\r\|_{L^{\infty}(d\mu)}.
$$
Then, one sees that 
$[w]_{A_{\infty}}\le[w]_{A_p}$ 
for $1<p<\infty$ and, 
using the dominated convergence theorem for conditional expectations, 
one sees also that 
$(\cE_i\sg)^{p-1}$ 
converges a.~e. to 
$\ds\exp\l(-\cE_i(\log w)\r)$. 

Corollary \ref{cor4.5} assert that 
there exists a constant $C_p>0$ such that 
$$
\|(\cdot)^{*}\|_{L^p(wd\mu)\rightarrow L^p(wd\mu)}
\le C_p
[w]_{A_p}^{\frac1{p-1}},
$$
where $C_p$ depends on $p$ but not on $w$.
Since 
$\ds[w]_{A_p}=[\sg]_{A_{p'}}^{p-1}$,
we have 
\begin{equation}\label{5.1}
\|(\cdot)^{*}\|_{L^p(wd\mu)\rightarrow L^p(wd\mu)}
\le C_p
\l([w]_{A_p}[\sg]_{A_{p'}}\r)^{\frac1p}.
\end{equation}
The following theorem sharpens \eqref{5.1}. 

\begin{theorem}\label{thm5.1}
$$
\|(\cdot)^{*}\|_{L^p(wd\mu)\rightarrow L^p(wd\mu)}
\le C_p
\l([w]_{A_p}[\sg]_{A_{\infty}}\r)^{\frac1p},
$$
where $C_p$ depends on $p$ but not on $w$.
\end{theorem}

\begin{proof}
Let $i\in\Z$ be arbitrarily chosen and fixed. 
By Theorem \ref{thm4.1}, 
we have to prove that , 
for any $E\in\cF_i^0$,
$$
\|1_{E}\sup_{j\ge i}\cE_j\sg\|_{L^p(wd\mu)}^p
\le C
[w]_{A_p}[\sg]_{A_{\infty}}
[\sg d\mu](E).
$$
Let us now apply the construction of principal set as follows. 

Since we have 
$$
\|1_{E}\sup_{j\ge i}\cE_j\sg\|_{L^p(wd\mu)}^p
=
\|1_{E}\sup_{j\ge i}\cE_j[1_{E}\sg]\|_{L^p(wd\mu)}^p,
$$
we may assume that 
$E=P_0$ satisfies 
$P_0\in\cF_i^0$, $\mu(P_0)>0$ 
and, for some $k\in\Z$, 
$$
2^{k-1} 1_{P_0}<\cE_i[1_{P_0}\sg]\le 2^k 1_{P_0}
$$
by a simple dyadic decomposition argument.
We write $\kp_1(P_0):=i$ and $\kp_2(P_0):=k$.
We let $\cP_1:=\{P_0\}$ which 
we call the first generation of principal sets. 
To get the second generation of principal sets we define a stopping time 
$$
\tau_{P_0}
:=
\inf\{j\ge i:\,\cE_j[1_{P_0}\sg]>2^{k+1} 1_{P_0}\}.
$$
We say that a set $P\subset P_0$ 
is a principal set with respect to $P_0$ 
if it satisfies $\mu(P)>0$ and 
there exist $j>i$ and $l>k+1$ such that 
$$
P
=
\{2^{l-1} 1_{P_0} <\cE_j[1_{\{\tau_{P_0}=j\}}\sg]\le 2^l 1_{P_0}\}.
$$
Noticing that such $j$ and $l$ are unique, 
we write $\kp_1(P):=j$ and $\kp_2(P):=l$.
We let $\cP(P_0)$ be the set of all principal sets with respect to $P_0$ 
and let $\cP_2:=\cP(P_0)$ which 
we call the second generation of principal sets. 

We now need to verify that 
\begin{equation}\label{5.2}
\mu(P_0)\le 2\mu(E(P_0))
\end{equation}
where
$$
E(P_0)
:=
P_0\cap\{\tau_{P_0}=\infty\}
=
P_0\setminus\bigcup_{P\in\cP(P_0)}P.
$$
Indeed, 
it follows from the use of weak-$(1,1)$ boundedness of Doob's maximal operator that 
$$
\mu(P_0\cap\{\tau_{P_0}<\infty\})
\le
2^{-k-1}\int_{P_0}\sg\,d\mu
=
2^{-k-1}\int_{P_0}\cE_i\sg\,d\mu
\le 2^{-1}\mu(P_0).
$$
This clearly implies \eqref{5.2}. 

The next generations are defined inductively,
$$
\cP_{n+1}
:=
\bigcup_{P\in\cP_n}\cP(P),
$$
and we define the collection of principal sets $\cP$ by 
$$
\cP:=\bigcup_{n=0}^{\infty}\cP_n.
$$
It is easy to see that 
the collection of principal sets $\cP$ satisfies the following properties:

\begin{enumerate}
\item[{\rm(i)}] 
The sets $E(P)$ where $P\in\cP$,
are disjoint and 
$P_0=\ds\bigcup_{P\in\cP}E(P)$;
\item[{\rm(ii)}] 
$P\in\cF_{\kp_1(P)}$;
\item[{\rm(iii)}] 
$\mu(P)\le 2\mu(E(P))$;
\item[{\rm(iv)}] 
$2^{\kp_2(P)-1} <\cE_{\kp_1(P)}\sg\le 2^{\kp_2(P)} $
on $P$;
\item[{\rm(v)}] 
$\ds\sup_{j\ge i}\cE_j[1_{P}\sg]
\le
2^{\kp_2(P)+1} $ on $E(P)$.
\end{enumerate}

We estimate as follows: 
\begin{alignat*}{2}
(*)
&:=
\|1_{P_0}\sup_{j\ge i}\cE_j[1_{P_0}\sg]\|_{L^p(wd\mu)}^p
\\ &=
\sum_{P\in\cP}
\|1_{E(P)}\sup_{j\ge i}\cE_j[1_{P_0}\sg]\|_{L^p(wd\mu)}^p
\\ &\le 2^p
\sum_{P\in\cP}
[wd\mu](E(P))2^{p\kp_2(P)}
\le 2^p2^{p-1}
\sum_{P\in\cP}
2^{\kp_2(P)}
[wd\mu](E(P))2^{(p-1)(\kp_2(P)-1)}
\\ &\le 2^p2^{p-1}
\sum_{P\in\cP}
2^{\kp_2(P)}
\int_{E(P)}\l(\cE_{\kp_1(P)}\sg\r)^{p-1}w\,d\mu
\le 2^p2^{p-1}
\sum_{P\in\cP}
2^{\kp_2(P)}
\int_{P}\l(\cE_{\kp_1(P)}\sg\r)^{p-1}w\,d\mu
\\ &= 2^p2^{p-1}
\sum_{P\in\cP}
2^{\kp_2(P)}
\int_{P}\l(\cE_{\kp_1(P)}\sg\r)^{p-1}\cE_{\kp_1(P)}w\,d\mu,
\end{alignat*}
where in the last two steps 
we have used $E(P)\subset P$ 
and (ii). 
The definition of $A_p$ and (iii) yield 
$$
(*)
\le 4^p[w]_{A_p}
\sum_{P\in\cP}
2^{\kp_2(P)}\mu(E(P)).
$$
Since the definition of $A_{\infty}$ and (iv) imply 
$$
2^{\kp_2(P)}
\le
2\cE_{\kp_1(P)}\sg
\le 2[\sg]_{A_{\infty}}
\exp\l(\cE_{\kp_1(P)}\log\sg\r)
\text{ on }E(P),
$$
we have further that 
\begin{alignat*}{2}
(*)
&\le
2\cdot 4^p[w]_{A_p}[\sg]_{A_{\infty}}
\sum_{P\in\cP}
\int_{E(P)}
\sup_{j\ge i}
\exp\l(\cE_j[\log1_{P_0}\sg]\r)
\,d\mu
\\ &=
2\cdot 4^p[w]_{A_p}[\sg]_{A_{\infty}}
\int_{P_0}
\sup_{j\ge i}
\exp\l(\cE_j[\log1_{P_0}\sg]\r)
\,d\mu.
\end{alignat*}
For any $q>1$ we have 
$$
\exp\l(\cE_j[\log1_{P_0}\sg]\r)
=
\l\{\exp\l(\cE_j[\log(1_{P_0}\sg)^{\frac1q}]\r)\r\}^q
\le
\l(\cE_j[(1_{P_0}\sg)^{\frac1q}]\r)^q
$$
by Jensen's inequality for conditional expectation. 
This yields 
$$
\sup_{j\ge i}
\exp\l(\cE_j[\log1_{P_0}\sg]\r)
\le
\l(
\sup_{j\ge i}
\cE_j[(1_{P_0}\sg)^{\frac1q}]
\r)^q.
$$
Finally, 
Doob's maximal inequality gives us that 
$$
(*)
\le
2\cdot 4^p[w]_{A_p}[\sg]_{A_{\infty}}
(q')^q
[\sg d\mu](P_0).
$$
Letting $q\rightarrow\infty$, we obtain 
$$
(*)
\le
2\cdot 4^p[w]_{A_p}[\sg]_{A_{\infty}}e
[\sg d\mu](P_0).
$$
This completes the proof. 
\end{proof}

\end{document}